\documentclass{article}
\usepackage{amssymb, amsmath, amsthm}
\usepackage{hyperref}

\newtheorem{theorem}{Theorem}    
          
\newtheorem{lem}{Lemma}          
\newtheorem{cnj}{Conjecture} 
\newcommand{\ir}{\mathcal{I}^{(r)}}
\newcommand{\F}{\mathcal{F}}
\newcommand{\Fsig}{\F_{[\sigma_1 \times \sigma_2]}}

\begin{document}
\title{An EKR Theorem for the Cartesian Product of Complete Graphs}
\author{
Zaphenath Joseph \footnotemark[3]}

\date{}

\maketitle
\footnotetext[3]{
Department of Mathematics \& Statistics, 
Villanova University, 
Villanova, PA, USA, 
\texttt{zjoseph@villanova.edu}
} 
\begin{abstract}
The Erdős--Ko--Rado theorem states that for \( r \leq \frac{n}{2} \), the largest intersecting family of $r$-subsets of $[n]$ is given by fixing a common element in all subsets, which trivially ensures pairwise intersection. We investigate this property for families of independent sets in the Cartesian product of complete graphs, $K_n \times K_m$. Using a novel extension of Katona's cycle method, we prove $K_n \times K_m$ is $r$-EKR when $1 \leq r \leq \frac{\min(m,n)}{2}$, demonstrating the Holroyd--Talbot conjecture holds for this class of well-covered graphs.
\end{abstract}
\section{Introduction}


For a positive integer $n$, let $[n] = \{1, 2, \dots, n\}$ and  $\binom{[n]}{r}$ denote the collection of all $r$-element subsets of $[n]$. We say that a family $\F \subseteq \binom{[n]}{r}$ is \emph{intersecting} if $F \cap F' \neq \emptyset$ for any $F, F' \in \F$. The Erd\H{o}s--Ko--Rado (EKR) theorem is a fundamental result in extremal set theory concerning intersecting families. 
\begin{theorem}[Erd\H{o}s, Ko, and Rado \cite{EKR}]
    For positive integers $n$ and $r$, let $r\leq\frac{n}{2}$. If $\F \subseteq \binom{[n]}{r}$ is an intersecting family, then 
    \[
    |\F| \leq \binom{n-1}{r-1}.
    \]
\end{theorem}
For $r<\frac{n}{2}$, equality in the bound is uniquely obtained by \emph{stars}: families where all sets contain a common fixed element. We call $\F$ a \emph{star centered at $x$} if every set in $\F$ contains $x$. This result was first proved using classical shifting techniques, but was later given an elegant alternative proof by Katona \cite{Katona} presenting his cycle method.

A natural graph-theoretic extension of the EKR theorem considers families of independent $r$-subsets of the vertex set of a graph. For a graph $G$, let $\ir(G)$ denote the collection of all independent $r$-subsets of $V(G)$, and let $\ir_v(G)$ denote the star centered at $v \in V(G)$. We say $G$ has the \emph{EKR property} (or is \emph{$r$-EKR}) if there exists a vertex $v \in V(G)$ such that for every intersecting family $\F \subseteq \ir(G)$,  $|\F| \leq |\ir_v(G)|.$

Let $\alpha(G)$ denote the size of the largest independent set in $G$ and $\mu(G)$ denote the size of the smallest maximal independent set. The following conjecture was made by Holroyd and Talbot regarding graphs with the EKR-property:

\begin{cnj}[Holroyd and Talbot \cite{HT}]
   A graph $G$ is \emph{r-EKR} for $1\leq r\leq\frac {\mu(G)}{2}.$
\end{cnj}
In \cite{HT}, the authors propose that \emph{well-covered} graphs, i.e., graphs $G$ that satisfy $\alpha(G) = \mu(G)$, form a natural class of potential counterexamples to the conjecture. In the same paper, the authors studied products of graphs. For graphs $G$ and $H$, define the \emph{lexicographic product of $G$ with $H$}, denoted $G[H]$, as the graph with vertex set $V(G)\times V(H)$ having edges $(x,y)(a,b)$ whenever $xa\in E(G)$ or $x=a$ and $yb\in E(H)$. The authors prove a statement regarding the  lexicographic product of a graph with $K_n$, the complete graph on $n$ vertices:
\begin{theorem}
 If a graph $G$ is \emph{r-EKR}, then $G[K_n]$ is \emph{r-EKR}.    
\end{theorem}

Note that the edge condition for the lexicographic product of graphs closely resembles that of the Cartesian product of graphs. The \emph{Cartesian product} $G \times  H$ is the graph with vertex set $V(G) \times V(H)$ and edge set 
\[
E(G \times  H) = \big\{ (g_1,h_1)(g_2,h_2) \;\big|\; (g_1 = g_2 \text{ and } h_1h_2 \in E(H)) \text{ or } (h_1 = h_2 \text{ and } g_1g_2 \in E(G)) \big\}.
\]
Despite being the most well-known graph product, the Cartesian product has seen limited study in the context of EKR properties. 

In his survey paper, Hurlbert \cite{hbert} considers the problem of determining whether the \emph{rook's graph}, $K_n \times K_m$, is $r$-EKR. The rook's graph  merits consideration for two fundamental reasons: it is a well-covered  graph which satisfies the Holroyd-Talbot conjecture, while its symmetric structure offers a tractable setting to examine EKR properties in graph products. 

One can visualize the rook's graph to be a $n$ by $m$ grid, where any two vertices in the same row or column are adjacent. Thus, any two vertices in an independent set must have distinct first coordinates and distinct second coordinates.

Note that 
\[ \big|\ir(K_n \times K_m)\big|= \binom{n}{r}\binom{m}{r} r!,\] 
which implies \[\big|\ir_v(K_n \times K_m)\big| = \binom{n-1}{r-1}\binom{m-1}{r-1} (r-1)!\]
for any $v \in V(K_n \times K_m)$ since the graph is vertex transitive.

\section{Main Result}

\begin{theorem}
For positive integers $ n\leq m$ and for $1\leq r\leq \frac{n}{2}, K_n \times K_m$ is \emph{r-EKR}.
\end{theorem}

\begin{proof}
Let $S_n$ be the symmetric group on $[n]$. Define equivalence relation $\sim$ on $S_n \times S_m$ by:
\[
(\sigma_1,\sigma_2) \sim (\psi_1,\psi_2) \iff \]\[\text{There exists } i\in [n] \text{ and } j\in{[m]} \text{ so that every } x \in [n] \text{ and } y \in [m],\ (\sigma_1(x),\sigma_2(y)) = (\psi_1(x+i), \psi_2(y+j))
\]
where addition is done mod $n$ and $m$ respectively. Call each equivalence class a \textit{cyclic order}. Since each pair of permutations has $nm$ equivalent pairs of permutations, and $|S_n \times S_m| = n!m!$, then  there are $(n-1)!(m-1)!$ cyclic orders. 

For a fixed cyclic order, a set $A \subseteq V(K_n\times K_m)$ is a \textit{$[\sigma_1\times \sigma_2]$-interval} if 
\[ 
A = \{(\sigma_1(i+k), \sigma_2(j+k)) : 0 \leq k \leq r-1\} \text{ for some } i \in [n] \text{ and }j\in[m].
\]
Define the subfamily:
\[
\Fsig = \{F \in \F : F \text{ is a $[\sigma_1\times\sigma_2]$-interval}\}.
\]

\begin{lem} \label{lem1}
    $|\Fsig| \leq r$
\end{lem}
\begin{proof}
Fix a cyclic order $[\sigma_1 \times \sigma_2]$ and let $\pi: V(K_n \times K_m) \to V(K_n)$ be the projection $\pi((a,b)) = a$. Extended $\pi$ to sets by $\pi(A) = \{a : (a,b) \in A\}$. 

Assume $\Fsig \neq \emptyset$ and fix $A_0 \in \Fsig$ starting at $(\sigma_1(x),\sigma_2(y))$:
\[
A_0 = \{(\sigma_1(x+k), \sigma_2(y+k)) : 0 \leq k \leq r-1\}.
\]

For $1 \leq i \leq r-1$, define the following families:
\[
\mathcal{A}^i =\{A \in \ir(K_n\times K_m) : \pi(A) = \{\sigma_1(x+i), \sigma_1(x+i+1), \dots, \sigma_1(x+i+r-1)\}\}
\]
\[
\mathcal{A}_i = \{A \in \ir(K_n\times K_m) : \pi(A) = \{\sigma_1(x+i-r), \sigma_1(x+i-r), \dots, \sigma_1(x+i-1)\}\}.
\]
Notice that for any two independent sets $A$ and $B$, if $A\cap B\neq \emptyset$, then $\pi(A)\cap \pi(B)\neq \emptyset$. So if $B\cap A_0\neq\emptyset$, then $B$ must be a set in some $\mathcal{A}_i $ or $\mathcal{A}^i $ as these are the only sets whose images under $\pi$ intersect with $\pi(A_0)$. But if $A_i\in \mathcal{A}_i$ and $A^i\in \mathcal{A}^i$, then $\pi(A_i)\cap \pi(A^i)=\emptyset$ which means that $A_i\cap A^i=\emptyset$. This gives us that $\Fsig$ can only contain sets from at most one of $\mathcal{A}_i $ or $\mathcal{A}^i$ for each $1\leq i\leq r-1$. The remainder of the proof follows from the fact that any two distinct sets 
in $\mathcal{A}_i$ are disjoint, and similarly for $\mathcal{A}^i$. This implies that $|\Fsig \cap \mathcal{A}_i|\leq 1$ and $|\Fsig \cap \mathcal{A}^i|\leq 1$

\end{proof}
For a fixed set $A \in \F$, $A$ appears as an interval in $r!(n-r)!(m-r)!$ cyclic orders. Combining this with Lemma \ref{lem1} and the fact that there are $(n-1)!(m-1)!$ cyclic orders give us
\[
|\F|  \space r!(n-r)!(m-r)! \leq r  (n-1)!(m-1)! \, ,\]
or equivalently,
\[
|\F| \leq\binom{n-1}{r-1}\binom{m-1}{r-1} (r-1)!. 
\]
\end{proof}


\end{document}